\documentclass[12pt]{article}
\usepackage{amsmath,amssymb, amsfonts, amsthm,amscd}
\usepackage[T2A]{fontenc}
\usepackage[cp1251]{inputenc}
\usepackage[english]{babel}
\usepackage{graphicx}

\theoremstyle{plain}
\newtheorem{thm}{Theorem}
\newtheorem{prop}{Proposition}

\newtheorem{rem}{Remark}

\theoremstyle{definition}
\newtheorem{defini}{Definition}
\newtheorem{defin}[defini]{Definition}

\renewcommand\refname{\large \textbf{References}}

\begin{document}
\begin{center}
{\huge A Bezout ring of stable range 2 which has square stable  range 1}
\end{center}
\vskip 0.1cm \centerline{{\Large Bohdan Zabavsky, \; Oleh Romaniv}}

\vskip 0.3cm

\centerline{\footnotesize{Department of Mechanics and Mathematics}}
\centerline{\footnotesize{Ivan Franko National University of Lviv,  Ukraine}}

\centerline{\footnotesize{zabavskii@gmail.com,\quad oromaniv@franko.lviv.ua}}
\vskip 0.5cm

\centerline{\footnotesize{November, 2018}}
\vskip 0.7cm

\footnotesize{\noindent\textbf{Abstract:} \textit{In this paper we introduced the concept  of a ring of stable range 2 which has square stable  range 1. We proved that a Hermitian ring $R$ which has (right) square stable range 1 is an elementary divisor ring if and only if $R$ is a duo ring of neat range 1. And we proved that a commutative Hermitian ring  $R$ is a Toeplitz ring if and only if $R$ is a ring of (right) square range 1. We proved that if $R$  be a commutative elementary divisor ring of (right) square stable range 1, then for any matrix $A\in M_2(R)$  one can find invertible Toeplitz matrices $P$ and $Q$ such that $
 PAQ=\left(\begin{smallmatrix}
 e_1&0\\
 0&e_2
 \end{smallmatrix}\right),
 $
 where $e_i$ is a divisor of~$e_2$.} }


\vspace{1ex}
\footnotesize{\noindent\textbf{Key words and phrases:} \textit{Hermitian ring, elementary divisor ring, stable range 1, stable range 2, square stable range 1, Toeplitz matrix, duo ring,  quasi-duo ring.}

}

\vspace{1ex}

\hskip 1,5 true cm

\normalsize

\section{Introduction}
\label{g1220}

    The notion of a stable range of a ring was introduced by H.~Bass, and became especially popular because of its various applications to the problem of cancellation and substitution of modules. Let us say that a module $A$ satisfies the power-cancellation  property if for all modules $B$ and $C$, $A\oplus B\cong A\oplus C$ implies that $B^n\cong C^n$ for some positive integer $n$ (here $B^n$ denotes the direct sum of $n$ copies of $B$). Let us say that a right $R$-module $A$ has the power-substitution property if given any right $R$-module decomposition $M=A_1\oplus B_1=A_2\oplus B_2$ which each $A_i\cong A$, there exist a positive integer $n$ and a submodule $C\subseteq M^n$ such that $M^n=C\oplus B_1^n=C\oplus B_2^n$.

    Prof. K.~Goodearl pointed out that a commutative rind $R$ has the power-substitution property if and only if $R$ is of (right) power stable range 1, i.e. if $aR+bR=R$ than $(a^n+bx)R=R$ for some $x\in R$ and some integer $n\ge 2$ depending on $a,b\in R$ \cite{Goodearl76}.

    Recall that a ring $R$ is said to have 1 in the stable range provided that whenever $ax+b = 1$ in $R$, there exists $y\in R$ such that $a+by$ is a unit in $R$. The following Warfield's  theorem shows that 1 in the stable range is equivalent to a substitution property.

\begin{thm} \cite{Goodearl76}
    Let $A$ be a right $R$-module, and set $E=\mathrm{End}_R(A)$. Then $E$ has 1 in the stable range if and only if for any right $R$-module decomposition $M=A_1\oplus B_1=A_2\oplus B_2$ with each $A_i\cong A$, there exists a submodule $C\subseteq M$ such that $M=C\oplus B_1=C\oplus B_2$.
\end{thm}

    A ring $R$ is said to have 2 in the stable range if  for any $a_1,\dots, a_r\in R$ where $r\ge 3$ such that $a_1R+\dots+a_rR=R$, there exist elements $b_1,\dots, b_{r-1}\in R$ such that $(a_1+a_rb_1)R+(a_2+a_rb_2)R+\dots+(a_{r-1}+a_rb_{r-1})R=R$.

    K.~Goodearl pointed out to us the following result.

\begin{prop} \cite{Goodearl76}
    Let $R$ be a commutative ring which has 2 in the stable range. If $R$ satisfies right power-substitution, then so does $M_n(R)$, for all $n$.
\end{prop}

    Our goal this paper is to study certain algebraic versions of the notion of stable range 1. In this paper we study  a Bezout ring which has 2 in the stable range and which is a ring square stable range 1.

    A ring $R$ is said to have (right) square stable range 1 (written $ssr(R)=1$) if $aR+bR=R$ for any $a,b\in R$ implies that  $a^2+bx$ is an invertible element of $R$ for some $x\in R$. Considering the problem of factorizing the matrix  $\left(\begin{smallmatrix}a&0\\b&0\end{smallmatrix}\right)$ into a product of two Toeplitz matrices. D.~Khurana, T.Y.~Lam and Zhou Wang were led to ask go units of the form $a^2+bx$ given that $aR+bR=R$.

    Obviously,  a commutative ring which has 1 in the stable range is a ring which has (right) square stable range 1, but not vice versa in general. Examples of rings which have (right) square stable range 1 are rings of continuous real-valued functions on topological spaces and real holomorphy rings in formally real fields~\cite{KhuLamWang11}.

\begin{prop} \cite{KhuLamWang11} \label{prop-2}
    For any ring $R$ with $ssr(R) = 1$, we have that $R$  is right quasi-duo (i.e. $R$ is a ring in which every maximal right ideal  is an ideal).
\end{prop}

    We say that matrices $A$ and $B$ over a ring $R$ are equivalent if there exist invertible matrices $P$ and $Q$ of appropriate sizes such that $B=PAQ$.   If for a matrix $A$ there exists a diagonal matrix $D=\mathrm{diag}(\varepsilon_1,\varepsilon_2,\dots,\varepsilon_r,0,\dots,0)$ such that $A$ and $D$ are equivalent  and $R\varepsilon_{i+1}R\subseteq \varepsilon_iR\cap R\varepsilon_i$ for every $i$ then we say that the matrix $A$ has {a canonical diagonal reduction. A ring $R$ is called an elementary divisor ring if every matrix over $R$ has a canonical diagonal reduction. If every $(1\times2)$-matrix ($(2\times 1)$-matrix) over a ring $R$ has a canonical diagonal reduction then $R$ is called a right (left) Hermitian ring. A ring which is both  right and left Hermitian  is called an Hermitian ring. Obviously, a commutative right (left) Hermitian ring is an Hermitian ring. We note that a right Hermitian ring is a ring in which every finitely generated right ideal is principal.

\begin{thm} \cite{ZabKom90} \label{thm-3}
    Let $R$ be a right quasi-duo elementary divisor ring. Then for any $a\in R$ there exists an element $b\in R$ such that $RaR=bR=Rb$. If in addition all zero-divisors of $R$ lie in the  Jacobson radical, then $R$ is a duo ring.
\end{thm}

    Recall that a right (left) duo ring is a ring in which every right (left) ideal is two-sided. A duo ring is a ring which is both left and right duo ring.

    We have proved the next result.

\begin{thm}\label{thm-4}
    Let $R$ be an elementary divisor ring which has (right) square stable range 1 and which all zero-divisors of $R$ lie in Jacobson radical of $R$, then $R$ is a duo ring.
\end{thm}
\begin{proof}
  By Proposition~\ref{prop-2} we have that $R$ is a right quasi-duo ring. By Theorem~\ref{thm-3} we have that $R$ is a duo ring.
\end{proof}

\begin{prop} \label{prop-5}
    Let $R$ be a Hermitian duo ring. For every $a,b,c\in R$ such that $aR+bR+cR=R$ the following conditions are equivalent:
    \begin{itemize}
      \item[1)] there exist elements $p,q\in R$ such that $paR+(pb+qc)R=R$;
      \item[2)] there exist elements $\lambda, u,v\in R$ such that $b+\lambda c=vu$, where $uR+aR=R$ and $vR+cR=R$.
    \end{itemize}
\end{prop}
 \begin{proof}
   1)$\Rightarrow$2) Since $paR+(pb+qc)R=R$ we have $pR+qcR=R$ and since $R$ is a duo ring we have $pR+cR=R$. Than $Rp+Rc=R$, i.e. $vp+jc=1$ for some elements $v,j\in R$. Then $vpb+jcb=b$ and $b-vpb=jcb=cj'b=ct$ where $t=j'b$ and $jc=cj'$. Element $j'$ exist, since $R$ is a duo ring.
      Then $v(pb+qc)=vpb+vqc=b+ct+vqc=b+ct+ck$, where $vqc=ck$ for some element $k\in R$.
      That is, we have $v(pb+qc)-b=c\lambda$ for some element $\lambda \in R$. We have $b+c\lambda =v(pb+qc)$. Let $u=pb+qc$.  We have $b+c\lambda =vu$, where $vR+cR=R$, since $vp+cj'=1$ and $uR+aR=R$, since $paR+(pb+qc)R=R$.

      2)$\Rightarrow$1) Since $vR+cR=R$ then $Rv+Rc=R$. Let $pv+jc=1$ for some elements $p,j\in R$. Then  $pR+cR=R$. Since $b+\lambda c=vu$, we have $pb=p(vu-\lambda c)=(pv)u-p\lambda c=(1-jc)u+p\lambda c=u-ju'c+p\lambda c=u+qc$ for some element $q=p\lambda -ju'$, where $cu=u'c$  for some element $u'\in R$. Since $u=pb+qc$, therefore $(pb+qc)R+aR=R$. Since $R$ is an Hermitian duo ring then we have $pR+qR=dR$ where $p=dp_1$, $q=dq_1$ and $p_1R+q_1R=R$. Then $p_1R+(p_1b+q_1c)R=R$ since $pR\subset p_1R$ and $pR+cR=R$, $p_1R+q_1R=R$, i.e. we have $p_1R+(p_1b+q_1c)R=R$. Hence, $aR+(p_1b+q_1c)R$ we have $p_1aR+(p_1b+q_1c)R=R$.
  \end{proof}

 \begin{rem}\label{rem-1}
 In  Proposition~\ref{prop-5} we can choose the elements $u$ and $v$ such that $uR+vR=R$.
 \end{rem}

 \begin{prop} \label{prop-6}
    Let $R$ be an Hermitian duo ring. Then the  following conditions are equivalent:
    \begin{itemize}
      \item[1)] $R$ is an elementary divisor duo ring;
      \item[2)]  for every $x,y,z,t\in R$ such that $xR+yR=R$ and  $zR+tR=R$ there exists an element $\lambda\in R$ such that $x+\lambda y=vu$, where $vR+zR=R$ and $uR+tR=R$.
    \end{itemize}
\end{prop}
 \begin{proof}
   1)$\Rightarrow$2) Let $R$ be an elementary divisor ring. By \cite{Zabavsk} for any $a$, $b$, $c$ such that $aR+bR+cR=R$ there exist elements $p,q\in R$ such that $paR+(pb+qc)R=R$.

   Since $xR+yR=R$, $zR+tR=R$ and the fact that $R$ is a Hermitian duo ring  we have $zR+xR+ytR=R$. By Proposition~\ref{prop-5} we have $x+\lambda yt=uv$ where $uR+zR=R$, $vR+ytR=R$. Since $x+(\lambda t)y=x+\mu y=uv$ where $\mu =\lambda t$, we have $uR+zR=R$, $vR+yR=R$.

      2)$\Rightarrow$1) Let $aR+bR+cR=R$ and $Rb+Rc=Rd$ and $b=b_1d$, $c=c_1d$, where $Rb_1=Rc_1=R$. Since $R$ is a duo ring then $b_1R+c_1R=R$. So now $dR=Rd$ and $aR+bR+cR=R$, $Rb+Rc=Rd$ we have $aR+dR=R$, i.e. $dd_1+ax=1$ for some elements $d_1,x\in R$. Then $1-dd_1\in aR$.

      Since $b_1R+c_1R=R$, by Conditions 2 of  Proposition~\ref{prop-5}  there exists an element $\lambda_1\in R$ such that $b_1+c_1\lambda=vu_1$ where $u_1R+(1-dd_1)R=R$ and $vR+dd_1R=R$. Since $(1-dd_1)\in aR$ and $u_1R+(1-dd_1)R=R$. We have $uR+aR=R$. Let $u=u_1d$. Since $u_1R+aR=R$ and $dR+aR=R$ we have $uR+aR=R$. Since $b_1+c_1\lambda=vu_1$, we have $b+c\mu+vu$,  where $\lambda d=d\mu$.

      Recall that $vR+dd_1R=R$ then $vR+dR=R$. Since $vR+cR=vR+c_1dR=vR+c_1R$. So  $b_1+c_1\lambda=vu_1$ and $b_1R+c_1R=R$ then $vR+c_1R=R$.

      Therefore, $vR+cR=R$. This means that the Condition 2 of Proposition~\ref{prop-5} is true. By Proposition~\ref{prop-5} we conclude that for every $a,b,c\in R$ with $aR+bR+cR=R$ there exist elements $p,q\in R$ such that $paR+(pb+qc)R=R$, i.e. according to \cite{Zabavsk}, $R$ is an elementary divisor ring.
  \end{proof}

\begin{defin}
    Let $R$ be a duo ring. We say that an element $a\in R\backslash \{0\}$ is neat if for any elements $b,c\in R$ such that $bR+cR=R$ there exist elements $r,s\in R$ such that $a=rs$, where $rR+bR=R$, $sR+cR=R$, $rR+sR=R$.
\end{defin}

\begin{defin}
We say that a duo ring $R$ has neat range 1 if for every $a,b\in R$ such that $aR+bR=R$ there exists an element $t\in R$ such that $a+bt$ is a neat element.
\end{defin}

According to Propositions \ref{prop-5}, \ref{prop-6} and Remark \ref{rem-1} we have the following  result.

\begin{thm}\label{thm-7}
    A Hermitian duo ring $R$ is an elementary divisor ring if and only if $R$ has neat range 1.
\end{thm}

The term "neat range 1" substantiates the following theorem.

\begin{thm}\label{thm-8}
    Let $R$ be a  Hermitian duo ring. If $c$ is a neat element of $R$ then $R/cR$ is a clean ring.
\end{thm}
\begin{proof}
    Let $c=rs$, where $rR+aR=R$, $sR+(1-a)R=R$ for any element $a\in R$. Let $\bar r=r+cR$, $\bar s=s+aR$. From the equality $rR+sR=R$ we have $ru+sv=1$ for some elements $u,v\in R$. Hence $r^2u+srv=r$ and $rsu+s^2v=s$ we have $\bar r^2\bar u=\bar r$, $\bar s^2\bar v=\bar s$. Let $\bar s\bar v=\bar e$. It is obvious that $\bar e^2=\bar e$ and $\bar 1-\bar e=\bar u\bar r$. Since $rR+aR=R$, we have $rx+ay=1$ for elements $x,y\in R$. Hence $rxsv+aysv=sv$ we have $rsx'v+aysv=sv$ where $xs=sx'$ for some element $x'\in R$. Then $\bar a\bar y\bar e=\bar e$, i.e. $\bar e\in\bar a\overline R$. Similarly  from the equality $sR+(1-a)R=R$, it follows $\bar 1-\bar e\in (\bar 1-\bar a)\overline R$. According to  \cite{Nicholson} $R/cR$ is an exchange ring. Since $R$  is a duo ring, $R/cR$ is a clean ring.
\end{proof}

Taking into account the Theorem~\ref{thm-4} and Theorem~\ref{thm-7} we have the following result.

\begin{thm}\label{thm-9}
    A Hermitian ring $R$ which has (right) square stable range 1 is an elementary divisor ring if and only if $R$ is a duo ring of neat range 1.
\end{thm}

    Let $R$ be a commutative Bezout ring. The matrix $A$ of  order 2 over $R$ is said to be a Toeplitz matrix if it is of the form
    $$
    \begin{pmatrix}
    a&b\\c&a
    \end{pmatrix}
    $$
    where $a,b,c\in R$.

    Notice that if $A$ is an invertible Toeplitz matrix, then $A^{-1}$ is also an invertible Toeplitz matrix.

 \begin{defin}
    A commutative Hermitian ring  $R$ is called a Toeplitz ring if for any $a,b\in R$  there exist an invertible Toeplitz matrix $T$ such that $(a,b)T=(d,0)$ for some element $d\in R$.
\end{defin}

\begin{thm}\label{thm-9-2}
    A commutative Hermitian ring  $R$ is a Toeplitz ring if and only if $R$ is a ring of (right) square range 1.
\end{thm}
\begin{proof}
    Let $R$ be a commutative Hermitian ring of (right) square stable range 1 and $aR+bR=R$ for some elements $a,b\in R$. Then $a^2+bt=u$, where $u$ is an  invertible element of $R$.

    Let
    $$
    S=    \begin{pmatrix}
    a&-b\\t&a
    \end{pmatrix}, \quad     K=\begin{pmatrix}
    u^{-1}&0\\0&u^{-1}
    \end{pmatrix}.
    $$
    Then
    $$
    (a,b)S=(u,0),\quad (u,0)K=(1,0),
    $$
    i.e. we have
    $$
    (a,b)SK=(1,0).
    $$
    Since
    $$
    \begin{pmatrix}
    a&-b\\t&a
    \end{pmatrix}\begin{pmatrix}
    u^{-1}&0\\0&u^{-1}
    \end{pmatrix}=\begin{pmatrix}
    au^{-1}&-bu^{-1}\\-tu^{-1}&au^{-1}
    \end{pmatrix}=T
    $$
    we have that $T=SK$ is a Toeplitz matrix. So $(a,b)T=(1,0)$. If $a,b\in R$ and $aR+bR=dR$ then by $a=da_0$, $b=db_0$ and $a_0R+b_0R=R$ \cite{Zabavsk}. Then there exists an element $t\in R$ such that $a_0+b_0t=u$, where $a$ is an invertible element of $R$.

    Let
    $$
    \begin{pmatrix}
    a_0&-b_0\\t&a_0
    \end{pmatrix}\begin{pmatrix}
    u^{-1}&0\\0&u^{-1}
    \end{pmatrix}.
    $$
    Note that $T$ is an invertible Toeplitz matrix. Then $(a,b)T=(d,0)$, i.e. $R$ is a  Toeplitz ring.

    Let $R$ be a  Toeplitz ring and $aR+bR=R$. The exists an invertible Toeplitz  matrix $T$ such that $(a,b)T=(1,0)$. Let $S=T^{-1}=\begin{pmatrix}x&t\\y&x\end{pmatrix}$, where $x,y,t\in R$. So $\det T^{-1}=z^2+ty=u$ is an invertible element of  $R$. Since $(a,b)=(1,0)T^{-1}$, we have $a=x$, $b=t$. By equality $x^2+ty=u$ we have $a^2+by=u$, i.e. $R$ is a ring of  (right) square stable range 1.
\end{proof}

\begin{thm}\label{thm-10}
    Let $R$ be a commutative ring of  square stable range 1. Then for any row $(a,b)$, where $aR+bR=R$, there  exists an invertible Toeplitz matrix
    $$
    T=\begin{pmatrix}a&b\\x&a\end{pmatrix},
    $$
    where $x\in R$.
\end{thm}
\begin{proof}
  By Theorem \ref{thm-9-2} we have $(a,b)=(1,0)T$ for some invertible Toeplitz matrix $T$. Let $
    T=\begin{pmatrix}x&t\\y&x\end{pmatrix}
    $. Then $a=x$, $b=t$ and $
    T=\begin{pmatrix}a&b\\y&a\end{pmatrix}
    $ is an invertible Toeplitz matrix.
\end{proof}

Recall that $GE_n(R)$ denotes a group of $n\times n$ elementary matrices over ring $R$. The following  theorem demonstrated that it is sufficient to consider only the case of matrices of order 2 in Theorem~\ref{thm-9-2}.

\begin{thm}\label{thm-11} \cite{Zabavsk}
 Let $R$  be a commutative elementary divisor ring. Then for any $n\times m$ matrix $A$ ($n>2$, $m>2$) one can find matrices $P\in GE_n(R)$ and $Q\in GE_m(R)$ such that
 $$
 PAQ=\begin{pmatrix}
 e_1&0&\dots&0&0\\
 0&e_2&\dots&0&0\\
 \dots&\dots&\dots&\dots&\dots\\
  0&0&\dots&e_s&0\\
  0&0&\dots&0&A_0
 \end{pmatrix}
 $$
 where $e_i$ is a divisor of $e_{i+1}$, $1\le i\le s-1$, and $A_0$ is a $2\times k$ or $k\times 2$ matrix for some $k\in\mathbb{N}$.
\end{thm}

\begin{thm}\label{thm-13}
 Let $R$  be a commutative elementary divisor ring of (right) square stable range 1. Then for any $2\times 2$ matrix $A$  one can find invertible Toeplitz matrices $P$ and $Q$ such that
 $$
 PAQ=\begin{pmatrix}
 e_1&0\\
 0&e_2
 \end{pmatrix},
 $$
 where $e_i$ is a divisor of $e_2$.
\end{thm}
\begin{proof}
  Since $R$ is a  Toeplitz ring it is enough to consider matrices of the form
  $$
  A=\begin{pmatrix}
  a&b\\0&c\end{pmatrix},
  $$
  where $aR+bR+cR=R$. Since $R$ is an elementary divisor ring by \cite{Zabavsk} there exist elements $p,q\in R$ such that $paR+(pb+qc)R=R$, i.e. $par +(pb+qc)s=1$ for some elements $r,s\in R$. Since $pR+qR=R$ and $rR+sR=R$, by Theorem~\ref{thm-10} we have the invertible Toeplitz matrices
  $P=\begin{pmatrix}p&q\\ * & *\end{pmatrix}$, $Q=\begin{pmatrix}r&*\\s&*\end{pmatrix}$ such that
  $$
  PAQ=\begin{pmatrix}1&x\\y&z\end{pmatrix}=A_1.
  $$
  Then
  $$
  \begin{pmatrix}1&0\\-y&1\end{pmatrix}A_1\begin{pmatrix}1&-x\\0&1\end{pmatrix}=\begin{pmatrix}1&0\\0&ac\end{pmatrix},
  $$
  where $S=\begin{pmatrix}1&0\\-y&1\end{pmatrix}$ and $T=\begin{pmatrix}1&-x\\0&1\end{pmatrix}$ are invertible Toeplitz matrices. So
  $$
  SPAQT=\begin{pmatrix}1&0\\0&ac\end{pmatrix}.
  $$
  Theorem is proved.
\end{proof}

\textbf{Open Question}. Is it true that every  commutative Bezout domain of stable range 2 which has (right) square stable  range 1 is an elementary divisor ring?

\vspace{5ex}
\renewcommand\refname{References}

\end{document}